\documentclass{amsart}

\usepackage{amsmath,amsthm,amssymb}
\usepackage{tikz}
\usetikzlibrary{arrows}

\newtheorem{theorem}{Theorem}[section]
\newtheorem{lemma}[theorem]{Lemma}
\newtheorem{proposition}[theorem]{Proposition}

\newtheorem{corollary}[theorem]{Corollary}

\newcommand{\fn}[1]{\mathrm{#1}}
\newcommand{\na}[1]{\mathsf{#1}}
\newcommand{\ax}[1]{(\mathsf{#1})}
\newcommand{\mdl}[1]{\mathcal{#1}}
\newcommand{\dash}{\text{-}}

\newcommand{\ph}{\varphi}

\newcommand{\QQ}{\mathbb{Q}}

\newcommand{\RR}{\mathbb{R}}

\newcommand{\st}{\; | \;}

\newcommand{\ML}{Martin-L\"of~}

\newcommand{\limplies}{\rightarrow}

\newcommand{\ex}[1]{\exists #1 \;} 
\newcommand{\fa}[1]{\forall #1 \;} 

\newcommand{\len}{\fn{length}}

\title[Randomness and the dominated convergence theorem]%
{Algorithmic randomness, reverse mathematics, and the dominated convergence theorem}

\author{Jeremy Avigad, Edward Dean, and Jason Rute}

\thanks{Work by the first and third authors has been partially supported by NSF grant DMS-1068829.}

\begin{document}

\begin{abstract}
 We analyze the pointwise convergence of a sequence of computable elements of $L^1(2^\omega)$ in terms of algorithmic randomness. We consider two ways of expressing the dominated convergence theorem and show that, over the base theory $\na{RCA_0}$, each is equivalent to the assertion that every $G_\delta$ subset of Cantor space with positive measure has an element. This last statement is, in turn, equivalent to weak weak K\"onig's lemma relativized to the Turing jump of any set. It is also equivalent to the conjunction of the statement asserting the existence of a 2-random relative to any given set and the principle of $\Sigma_2$ collection.
\end{abstract}

\maketitle

\section{Introduction}

Fix a measure space $\mdl X = (X, \mdl B, \mu)$. The dominated convergence theorem states that if $(f_n)$ is any sequence of integrable functions dominated by an integrable function $g$, and $(f_n)$ converges pointwise almost everywhere to a function $f$, then $f$ is an integrable function as well, and $(\int f_n)$ converges to $\int f$.

In the context of both computable measure theory and reverse mathematics, in the case where $X$ is a compact separable metric space and $\mdl B$ is the collection of Borel subsets of $\mdl X$, the space $L^1(\mdl X)$ of integrable functions modulo a.e.~equivalence can be represented as the completion of a countable set of test functions under the $L^1$ norm. There are then two ways that the dominated convergence theorem can be expressed in the language of second-order arithmetic, depending on whether the pointwise limit is assumed or asserted to exist as an element of $L^1(\mdl X)$. One option is to say that given a sequence $(f_n)$ of elements of $L^1(\mdl X)$ and an element $g$ of $L^1(\mdl X)$, if $(f_n)$ is dominated by $g$ and is pointwise convergent a.e., then there is an element $f$ of $L^1(\mdl X)$ such that $(f_n)$ converges to $f$ pointwise a.e.~and $(\int f_n)$ converges to $\int f$. The second option is to assume the existence of the limit, $f$, in advance, and say that  given $f$, $g$, and a sequence $(f_n)$ of elements of $L^1(\mdl X)$, if $(f_n)$ is dominated by $g$ and converges pointwise a.e.~to $f$, then $(\int f_n)$ converges to $\int f$. Let us call the first version $\ax{DCT}$ and the second version $\ax{DCT'}$.

Yu \cite{yu:94} has shown that, over $\na{RCA_0}$, $\ax{DCT}$ is equivalent to the arithmetic comprehension principle, $\ax{ACA}$. The implication from $\ax{DCT}$ to $\ax{ACA}$ is not difficult: if we take $\mdl X$ to be the unit interval $[0,1]$ under Lebesgue measure and take each $f_n$ to be a constant function, $\ax{DCT}$ implies, for example, that every monotone bounded sequence of rationals has a limit, a fact that easily implies $\ax{ACA}$ \cite{simpson:99}.

The status of $\ax{DCT'}$ remained open, however. Yu \cite{yu:94} showed that the corresponding formulation of the monotone convergence theorem is equivalent to a principle, weak weak K\"onig's lemma $\ax{WWKL}$, introduced by Yu and Simpson \cite{yu:simpson:90}. Simpson \cite{simpson:99} conjectured that $\ax{DCT'}$ is also equivalent to $\ax{WWKL}$. Our main result here is that, over $\na{RCA_0}$, $\ax{DCT'}$ is equivalent to a principle, $\ax{2\dash POS}$, which is strictly stronger than $\ax{WWKL}$, strictly weaker than $\ax{ACA}$, and incomparable with $\ax{WKL}$. $\ax{2\dash POS}$ asserts that any $G_\delta$ subset of Cantor space with positive measure has an element, and is equivalent to the relativization of $\ax{WWKL}$ to the Turing jump of any set. $\ax{2\dash POS}$ implies the statement $\ax{2\dash RAN}$ that there exists a 2-random relative to any given set. In fact, we show that over $\na{RCA_0}$, $\ax{2\dash POS}$ is equivalent to the conjunction of $\ax{2\dash RAN}$ and the principle of $\Sigma_2$ collection, $\ax{B\Sigma_2}$.

There is a more natural way of formulating the dominated convergence theorem without having to assert the existence of a pointwise limit. Namely, one says that given a sequence $(f_n)$ dominated by $g$, if $(f_n(x))$ is a Cauchy sequence for almost every $x$, then $(\int f_n)$ is Cauchy as well. Call this version $\ax{DCT^*}$. It is not hard to show that $\ax{DCT^*}$ implies $\ax{DCT'}$ over $\na{RCA_0}$, since whenever $(f_n(x))$ converges to $f(x)$ the sequence $f_0(x), f(x), f_1(x), f(x), f_2(x), \ldots$ is Cauchy. We show that, in fact, over $\na{RCA_0}$ the principle $\ax{DCT^*}$ is also equivalent to $\ax{2\dash POS}$. 

Below, we assume familiarity with some of the basic notions of algorithmic randomness \cite{downey:hirschfeldt:10,nies:09} and computable measure theory  \cite{galatolo:et:al:unp,hoyrup:rojas:09b,weihrauch:99}. We also assume familiarity with reverse mathematics, and, in particular, the formalization of measure-theoretic notions in subsystems of second-order arithmetic \cite{simpson:99,yu:simpson:90,yu:93,yu:94}. 

The outline of this paper is as follows. In Sections~\ref{randomness:section} and \ref{formalization:section}, we focus on Cantor space, $2^\omega$, as the natural setting for the study of algorithmic randomness. Section~\ref{randomness:section} considers measure-theoretic convergence statements in those terms. Section~\ref{formalization:section} begins to develop a framework for treating algorithmic randomness formally, in the context of subsystems of second-order arithmetic. These two strands come together in Section~\ref{reverse:section}, which provides a formal analysis of the dominated convergence theorem in terms of 2-randomness. In Section~\ref{n:randomness:section} we observe that the framework of Section~\ref{formalization:section} can be extended straightforwardly to deal with $n$-randomness, for every $n$.

The theory of 2-random subsets of $2^\omega$ has recently been brought to bear on reverse mathematics in an interesting way by Csima and Mileti \cite{csima:mileti:09}, who use it to build a model of the ``rainbow Ramsey theorem'' for pairs, $\ax{RRT^2_2}$, which is not a model of Ramsey's theorem for pairs. More recently, Theodore Slaman and Chris Conidis have announced that the argument of Csima and Mileti can be used to derive $\ax{RRT^2_2}$ from the axiom $\ax{2\dash RAN}$ we consider here. They have also shown that $\ax{2\dash RAN}$ is conservative over $\na{RCA_0 + \ax{B\Sigma_2}}$ for $\Pi^1_1$-sentences, and have studied its first-order consequences.

In an earlier version of this paper, we used the principle of $\Sigma_2$ induction, $\ax{I\Sigma_2}$, to prove $\ax{2\dash POS}$ from $\ax{2\dash RAN}$. We are grateful to Slaman for showing us a proof that only requires $\ax{B\Sigma_2}$, thereby strengthening our results; see Theorem~\ref{eq:thm:one} and Proposition~\ref{wwkl:ran:prop} below.

Kjos-Hanssen, Solomon, and Miller \cite{kjos:hanssen:et:al:unp} have considered a related principle, $\ax{POS}$, which asserts that every $G_\delta$ set of positive measure contains a closed set of positive measure. (The similarity of the name $\ax{2\dash POS}$ is coincidental.) Over $\na{RCA_0}$,  $\ax{POS} + \ax{WWKL}$ clearly implies $\ax{2\dash POS}$, while Cholak, Greenberg, and Miller \cite{cholak:et:al:06} have shown that $\ax{POS}$ does not imply $\ax{WWKL}$ and $\ax{WWKL} + \ax{POS}$ does not imply $\ax{WKL}$. We are grateful to Joseph Miller for calling our attention to this.

\section{Convergence and algorithmic randomness}
\label{randomness:section}

An element $\alpha$ of Cantor space, $2^\omega$, can be viewed as a one-way infinite binary sequence, but can also be identified with the set $X$ of natural numbers with characteristic function $\alpha$. A basis for the standard topology is given by the collection of sets of the form $[\sigma]$, where $\sigma$ is a finite binary sequence and $[\sigma]$ is the set of elements of $2^\omega$ that extend it. A \emph{name} (or \emph{code}) for an open subset of $2^\omega$ is a sequence $(B_i)_{i \in \omega}$, intended to denote $\bigcup_i B_i$, where each $B_i$ denotes a basic open set $[\sigma]$ or $\emptyset$. A subset $A$ of $2^\omega$ is \emph{computably open}, or $\Sigma^0_1$, if it has a computable name. The definition relativizes to any set $X$, so that open subsets of $2^\omega$ are exactly the ones that are $\Sigma^{0,X}_1$ for some $X$.

A $\Pi^0_1$ set is the complement of a $\Sigma^0_1$ set. These notions extend to the full arithmetic hierarchy; for example, a $\Pi^0_2$ set is of the form $\bigcap_i \bigcup_j B_{i,j}$, where $B_{i,j}$ is a uniformly computable sequence of basic open sets. Equivalently, one can view a $\Pi^0_2$ set as given by a computable sequence $(G_i)$ of (indices of) $\Sigma^0_1$ sets, or as given by a $\Pi^0_2$ formula in the language of arithmetic with free set variable $X$. It is not hard to pass back and forth between these representations \cite[Section 2.19]{downey:hirschfeldt:10}.

Now consider the usual coin-flipping measure $\mu$ on $2^\omega$ given by $\mu([\sigma]) = 2^{-\len(\sigma)}$. A \emph{\ML test}, or a \emph{\ML null set}, is a $\Pi^0_2$ set $\bigcap_i G_i$ which moreover has the property that $\mu(G_i) < 2^{-i}$; in other words, it is an effective sequence of $\Sigma^0_1$ sets whose measures converge to $0$ with an explicit rate of convergence. It is well known that an effective countable union of \ML tests is contained in a \ML test; in fact, there is a maximal one, called the \emph{universal} \ML test. An element of $2^\omega$ is \emph{\ML random}, or \emph{1-random}, if it is not an element of any \ML test, which is equivalent to saying that it is not an element of the universal \ML test. More generally, a \emph{$\Sigma^0_n$-test} is an effective sequence of $\Sigma^0_n$ sets whose measures converge to $0$ with the rate of convergence above, and an element of $2^\omega$ is \emph{$n$-random} if it is not in the universal $\Sigma^0_n$-test.

Let us turn to $L^1(2^\omega)$. A \emph{simple function} on $2^\omega$ is a function $f : 2^\omega \to \RR$ of the form $\sum_{i<n} a_i 1_{[\sigma_i]}$ where each $a_i \in \QQ$ and $1_{[\sigma_i]}$ denotes the characteristic function of $[\sigma_i]$. A \emph{name} of an element of $L^1(2^\omega)$ is a Cauchy sequence $(f_i)$ of (names of) simple functions such that for every $j \geq i$, $\| f_i - f_j \|_1 < 2^{-i}$.As before, an element $f$ of $L^1(2^\omega)$ is \emph{computable} if it has a computable name. One can show that if $(f_i)$ is a name of an element of $L^1(2^\omega)$, then $(f_i(x))$ converges for all $x$ outside a \ML null set, and that two computable names for the same element of $L^1(2^\omega)$ take the same value on \ML random points. (See \cite[Lemma 3.2]{pathak:09} and \cite[Section 4.2]{hoyrup:rojas:09x}, as well as \cite[Lemma 2.1]{yu:94}, which carries out the argument formally in $\na{RCA_0}$.) When we write $f(x)$ we mean to imply that this value is defined, which is to say, $\lim_i f_i(x)$ exists; and $f(x)$ then refers to this limit.

It will be convenient to blur the distinction between computable elements of $L^1(2^\omega)$ and their names. When we say that $(f_i)$ is a computable sequence of elements of $L^1(2^\omega)$, we mean that each $f_i$ is a computable element of $L^1(2^\omega)$ given by a sequence of names that are computable uniformly in $i$. When we say that $(f_i(x))$ converges we mean to imply that, moreover, $(f_i(x))$ is defined for every $i$.

\begin{theorem}
\label{effective:convergence:thm}
 Let $(f_i)$ be a computable sequence of elements of $L^1(2^\omega)$. Then up to a Martin-L\"of null set, the set of points $x$ such that the sequence $(f_i(x))$ converges is a $\Pi^0_3$ set. Similarly, if $f$ is a computable element of $L^1(2^\omega)$, then the set of points $x$ such that the sequence $(f_i(x))$ converges to $f(x)$ is a $\Pi^0_3$ set.
\end{theorem}

\begin{proof}
  Notice that in the last claim we can assume without loss of generality that $f = 0$, by considering the sequence $(f_i - f)_{i \in \omega}$. One can show that it is possible to effectively replace each $f_i$ with a simple function $f'_i$ without changing the limiting behavior of $f_i(x)$ on more than a \ML null set; see \cite[Lemma 3.1]{yu:94} and \cite{rute:unp}. (In the proof of Theorem~\ref{main:thm} below, it will be important to recognize that this argument goes through in $\na{RCA_0}$.) So, up to a \ML null set, the set of points $x$ for which $(f_i(x))$ converges is equal to $\{ x \st \fa {\varepsilon > 0} \ex m \fa {n > m} |f'_n(x) - f'_m(x)| \leq \varepsilon \}$, clearly a $\Pi^0_3$ set. Similarly, the set of points $x$ for which $(f_i(x))$ converges to $0$ is equal to $\{ x \st \fa {\varepsilon > 0} \ex m \fa {n > m} |f'_n(x)| \leq \varepsilon \}$.
\end{proof}

In the statement of the next corollary, $x$ is \emph{weakly 2-random}, by definition, if it is not in any null $\Pi^0_2$ set. (So every 2-random element of $2^\omega$ is weakly 2-random, and every weakly 2-random element is 1-random.) Similar considerations appear in Brattka, Miller, and Nies \cite{brattka:miller:nies}.

\begin{corollary}
 If $(f_i)$ is a computable sequence of elements of $L^1(2^\omega)$ that is pointwise a.e.~convergent to $0$, then $(f_i(x))$ converges to $0$ for every weak 2-random $x$.
\end{corollary}

\begin{proof}
By Theorem~\ref{effective:convergence:thm}, the set of points $x$ for which $(f_i(x))$ doesn't converge is contained in a null $\Sigma^0_3$ set, and hence a countable union of null $\Pi^0_2$ sets.
\end{proof}

\begin{corollary}
\label{eff:cor}
  If $(f_i)$ is any computable sequence of elements of $L^1(2^\omega)$ that is not pointwise a.e.~convergent to $0$, then there is a $\Pi^0_2$ set $A$ with positive measure such that $(f_i(x))$ does not converge to $0$ for any $x$ in $A$.
\end{corollary}

\begin{proof}
  Considering the $\Sigma^0_3$ set promised by Theorem~\ref{effective:convergence:thm} minus the \ML null set of exceptions, we see that the hypothesis implies that the set of $x$ such that $(f_i(x))$ doesn't converge contains a $\Sigma^0_3$ set with positive measure. But this is a countable union of $\Pi^0_2$ sets, one of which has to have positive measure.
\end{proof}


We will see in Section~\ref{formalization:section} that if $A$ is any $\Pi^0_2$ set with positive measure, then any 2-random element of $2^\omega$ computes an element of $A$. We will exploit this fact, together with Corollary~\ref{eff:cor}, in Section~\ref{reverse:section}, to show that over a suitable base theory the existence of 2-random elements of $2^\omega$ (relative to any set) implies the dominated convergence theorem. Roughly speaking, assuming that the conclusion of the dominated convergence theorem fails, we will produce an explicit $\Pi^0_2$ set with positive measure, any element of which provides a counterexample to the hypothesis.

In the other direction, to show that the dominated convergence theorem implies the existence of 2-randoms, it suffices to show that the dominated convergence theorem implies that every $\Pi^0_2$ set with positive measure has an element (since the complement of the universal $\Sigma^0_2$ test is the union of such sets). This will involve formalizing the following theorem.

\begin{theorem}
\label{two:pos:thm}
 Let $A$ be a $\Pi^0_2$ set with measure greater than $\delta$, for some $\delta > 0$. Then there is a sequence of simple characteristic functions $f_i$ such that $\int f_i > \delta$ for each $i$, but $(f_i(x))$ converges to $0$ for every $x$ outside of $A$.  
\end{theorem}

\begin{proof}
 Let $A = \bigcap_i G_i$, where each $G_i = \bigcup_j B_{i,j}$ is open with measure greater than $\delta$ and $G_0 \supseteq G_1 \supseteq \ldots$. Let $G'_i = \bigcup_{j \leq k} B_{i,j}$ for the least $k$ integer making the measure of this set greater than $\delta$. Let $f_i = 1_{G'_i}$. Then for each $i$, $\int f_i > \delta$. On the other hand, if $x \not\in A$, then $x \not\in G_i$ for some $i$, in which case $f_{i'}(x) = 0$ for every $i' \geq i$.
\end{proof}

All of the theorems and corollaries in this section relativize to an arbitrary set. In the next section, we will make these relativizations explicit. Section~\ref{reverse:section} deals with the same measure-theoretic notions in the context of a more general set of finite measure spaces. All the results described here hold in that more general setting, and the proofs can be adapted straightforwardly.

\section{Formalizing 2-randomness and related notions}
\label{formalization:section}

We now begin to provide a framework for the study of algorithmic randomness in the context of reverse mathematics. This involves importing the definitions in the last section to the language of second-order arithmetic. Specifically, we say (a code for) a \emph{$\Sigma^{0,X}_1$ subset $A$ of $2^\omega$} is (an index of) a sequence $(B_i)_{i \in \omega}$ of basic open sets, computable from $X$. With respect to the usual development of topological notions in reverse mathematics, an open subset of Cantor space is just a $\Sigma^{0,X}_1$ set for some $X$. We can then view a $\Pi^{0,X}_1$ set as the complement of a $\Sigma^{0,X}_1$ set. Similarly, a $\Pi^{0,X}_2$ subset of $2^\omega$ is a doubly-indexed sequence $B_{i,j}$ of basic open sets, viewed as $\bigcap_i \bigcup_j B_{i,j}$. These correspond to the $G_\delta$ subsets of Cantor space that are computable from $X$.

In the language of second-order arithmetic we can take $Y \in \bigcup_i [\sigma_i]$ to mean $\ex i (\sigma_i \subset Y)$, where the notation $\sigma_i \subset Y$ means that $\sigma_i$ is an initial segment of the characteristic function of $Y$. More generally, membership of $Y$ in a $\Sigma^{0,X}_n$ (resp.~$\Pi^{0,X}_n)$ set can be expressed by a $\Sigma^{0,X}_n$ (resp.~$\Pi^{0,X}_n$) formula with the additional parameter $Y$. But it is important to keep in mind that such a set is an \emph{intensional object}: it is a \emph{description} of a set of subsets of $\omega$, rather than the set itself. In particular, different ``sets'' $A$ can represent the same subset of $2^\omega$. 

As usual, the measure of an open set $\bigcup_i B_i$ is $\lim_n \mu(\bigcup_{i < n} B_i)$. When it comes to weak theories of reverse mathematics, however, one has to be careful, for at least three reasons. First, a weak theory cannot prove that the measure of an open set always exists; indeed, Yu \cite{yu:93} shows that this is equivalent, over the base theory $\na{RCA_0}$, to the principle $\ax{ACA}$ of arithmetic comprehension. Second, a weak theory cannot prove that extensionally equivalent descriptions of an open set have the same measure. For example, Yu and Simpson \cite{yu:simpson:90} show that the statement ``if $\bigcup_i B_i = 2^\omega$ then $\mu(\bigcup_i B_i) = 1$'' is equivalent to the principle $\ax{WWKL}$ discussed below, which is not provable in $\na{RCA_0}$. Finally, different characterizations of the measure of a set need not coincide; for example, if one defines the measure of a closed set in terms of its complement, in weak theories one cannot show that the measure of a closed set is the infimum of the measures of open sets covering it. Yu \cite{yu:93} shows that $\na{ACA_0}$ proves the existence and regularity of measures of sets at any finite level of the Borel hierarchy, eliminating all three problems in that axiomatic context. But in the context of the theories discussed here, the reader should keep in mind that all the definitions below are ``intensional'' and don't generally presuppose the existence of measures.

The principle $\ax{WWKL}$ introduced by Yu and Simpson \cite{yu:simpson:90} is as follows:
\begin{multline*}
 \fa T (\mbox{if $T$ is an infinite binary tree and} \\
  \lim_{n \to \infty} \frac{|\{\sigma \in T \st \len(\sigma) = n\}|}{2^n} > 0,
  \mbox{there is a path through $T$}).
\end{multline*}
If $A$ is a $\Sigma^{0,X}_1$ set $\bigcup_i B_i$ and $\delta \in \QQ$, then $\mu(A) > \delta$ is defined to be the assertion $\ex m (\mu(\bigcup_{i < m} B_i) > \delta)$. Notice that this is a $\Sigma^0_1$ formula in $X$, $A$, and $\delta$. Similarly,  $\mu(A) \leq \delta$ is the $\Pi^0_1$ assertion $\fa m (\mu(\bigcup_{i < m} B_i) \leq \delta)$. If $A$ is a $\Pi^{0,X}_1$ set, we can also express ``$\mu(A) < \delta$'' and ``$\mu(A) \geq \delta$,'' respectively, as $\Sigma^0_1$ and $\Pi^0_1$ formulas in $X$, $A$, and $\delta$. Let $\ax{1\dash POS}$ be the statement
\[
 \fa {X, A \in \Pi^{0,X}_1, \delta > 0} (\mu (A) \geq \delta \limplies \ex Y (Y \in A)).
\]
This expresses the statement that every closed set with positive measure has an element. Finally, define a \emph{Martin-L\"of test relative to $X$} to be a uniformly computable sequence $(G_i)_{i \in \omega}$ of $\Sigma^{0,X}_1$ sets such that for each $i$, $\mu(G_i) \leq 2^{-i}$. A set $Y$ is \emph{1-random relative to $X$} if for every Martin-L\"of test $(A_i)$ relative to $X$, $Y \not\in \bigcap_i A_i$, that is, $\ex i (Y \not\in A_i)$. Let $\ax{1\dash RAN}$ be the following principle:
\[
 \fa X \ex Y (\mbox{$Y$ is 1-random relative to $X$}).
\]
The fact that for any $X$ there is a universal Martin-L\"of test $(U_i)$ relative to $X$ can be proved straightforwardly in $\na{RCA_0}$. So, if $(U_i)$ is any such test, $Y$ is 1-random relative to $X$ if and only if $Y$ is not in $U_i$ for some $i$.

In the context of formal theories of arithmetic, recall the collection principles $\ax{B\Sigma_n}$:
\[
 \fa {x < u} \ex y \ph(x,y) \limplies \ex v \fa {x < u} \ex {y < v} \ph(x,y),
\]
where $\ph$ is any $\Sigma_n$ formula, possibly with number and set parameters other than $x$ and $y$ (see \cite{hajek:pudlak:93,simpson:99}). Let $\ax{I\Sigma_n}$ denote $\Sigma_n$ induction. Over a weak theory, $\ax{I\Sigma_n}$ implies $\ax{B\Sigma_n}$, and $\ax{B\Sigma_{n+1}}$ implies $\ax{I\Sigma_n}$. In particular, $\na{RCA_0}$ proves $\ax{B\Sigma_1}$, which suffices to show that the set of $\Sigma_1$ formulas is closed under bounded quantification. We will make use of these facts below.

\begin{theorem}
\label{eq:thm:one}
Over $\na{RCA_0}$, the following are equivalent:
\begin{enumerate}
 \item $\ax{WWKL}$
 \item $\ax{1\dash POS}$
 \item $\ax{1\dash RAN}$
\end{enumerate}
\end{theorem}

\begin{proof}
 The equivalence of $\ax{WWKL}$ and $\ax{1\dash POS}$ is proved by Yu and Simpson \cite{yu:simpson:90}. (In fact, Yu and Simpson prove that the conclusion holds for a wider classes of measure spaces; we will return to this in the next section.) So let us focus on the equivalence of $\ax{1\dash POS}$ and $\ax{1\dash RAN}$, which is asserted in \cite{simpson:09} without proof.

 First, suppose $\ax{1\dash POS}$. Given $X$, let $(U_i)$ be the universal \ML test relative to $X$. Then the complement of $U_1$ is a $\Pi^{0,X}_1$ set with positive measure, and any element of this set is 1-random relative to $X$.

 Conversely, assume $\ax{1\dash RAN}$. Let $C$ be a $\Pi^{0,X}_1$ set with positive measure, and let $\overline C = \bigcup_i [\sigma_i]$ be its complement. Without loss of generality, we can assume the sets $[\sigma_i]$ are disjoint. Let $\delta$ be a rational number less than 1 such that $\mu(\overline C) < \delta$. 

 Notice that, fixing a primitive recursive pairing function on the natural numbers, we can think of any set $W$ of natural numbers as coding a sequence $(W_i)_{i \in \omega}$ of sets of natural numbers, where $j \in W_i$ if and only if $(i, j) \in W$. For each $i$, write $W_i = \pi_i W$. If $\sigma$ is any finite binary sequence, then $\pi^{-1}_i [\sigma]$ is a finite union of cylinder sets, easily computable from $[\sigma]$. Moreover, the measure of $\pi^{-1}_i [\sigma]$ is clearly equal to the measure of $[\sigma]$, since the condition $U \in \pi^{-1}_i [\sigma]$ imposes $\len(\sigma)$-many constraints on the bits of $U$. Moreover, for any $\sigma$, $\tau$, and $i \neq j$, the sets $\pi^{-1}_i [\sigma]$ and $\pi^{-1}_j [\tau]$ are independent, which is to say, $\mu(\pi^{-1}_i [\sigma] \cap \pi^{-1}_j [\tau]) = \mu([\sigma]) \mu([\tau])$. This extends to finite unions of basic open sets: if $D$ and $E$ are such sets and $i \neq j$, then $\mu(\pi^{-1}_i D \cap \pi^{-1}_j E) = \mu(D) \mu(E)$.

 Returning to the proof, using $\ax{1\dash RAN}$, let $Y$ be 1-random relative to $X$. As above, write $Y_i = \pi_i Y$. It suffices to show that for some $i$, $Y_i$ is in $C$. Our proof will implicitly use the fact that each $Y_i$ is 1-random relative to $X$. The idea is to show that if each $Y_i$ is in $\overline C$, then, because the measure of $\overline C$ is less than 1, $Y$ itself is contained in a sequence of arbitrarily small open sets.

 In more detail, suppose that for every $i$, $Y_i$ is in $\overline C$. Then for every $i$ there is a $j$ such that $Y_i$ is in $[\sigma_j]$. Thus, by $\ax{B\Sigma_1}$, we have
\[
 \fa n \ex k \fa{i \leq n} \ex{j \leq k} Y_i \not\in [\sigma_j].
\]
 For each $n$, let $G_n = \{ W \st \fa {i \leq n} (W_i \in \overline C) \}$. Then $G_n$ is an open set, since we can write
\begin{align*}
 G_n & = \{ W \st  \ex k \fa{i \leq n} \ex{j \leq k} W_i \not\in [\sigma_j] \} \\
  & = \bigcup_k \bigcap_{i \leq n} \bigcup_{j \leq k} \pi_i^{-1} ([\sigma_j]) \\
  & = \bigcup_k \bigcap_{i \leq n} \pi_i^{-1}\left( \bigcup_{j \leq k}[\sigma_j] \right),
\end{align*}
 and for each $n$, $Y$ is in $G_n$. Moreover, by the observations above, for each $k$ we have
\[
 \mu\left(\bigcap_{i \leq n} \pi_i^{-1}\left( \bigcup_{j \leq k}[\sigma_j] \right) \right) \leq \delta^n,
\]
 and so $\mu(G_n) \leq \delta^n$. Thinning the sequence $(G_n)$ to a \ML test, we have a contradiction to the fact that $Y$ is 1-random relative to $X$.
\end{proof}

In an earlier draft, we used a formalization of Ku{\v{c}}era's theorem \cite{kucera:85} (see also \cite[Section 6.10]{downey:hirschfeldt:10}) to prove that $\ax{1\dash RAN}$ implies $\ax{1\dash POS}$. This, in turn, required the use of $\Sigma_1$ induction. We are grateful to Theodore Slaman for showing us the proof above, which uses only $\Sigma_1$ collection (so, in fact, the equivalence goes through in the system $\na{RCA_0^*}$ of Simpson and Smith \cite{simpson:smith:86,simpson:99}). This also enabled us to strengthen the statements of Proposition~\ref{wwkl:ran:prop} and Theorem~\ref{main:rand:theorem} below.

Our goal now is to carry out a similar analysis of 2-randomness, as well as the assertion that there is an element of any $\Pi^{0,X}_2$ set with positive measure and an analogue of $\ax{WWKL}$ which involves trees computable from $X'$. But there are two fine points that need to be addressed: first, how to say that a $\Pi^{0,X}_2$ set has positive measure, in light of the warnings above; and second, how to refer to $X'$ when the existence of Turing jumps is not provable in the weak theories we are considering here.

The second concern is easily met: simply use an appropriate $\Sigma^{0,X}_1$ formula to describe the Turing jump of $X$. We can express the fact that Turing machine $e$ with oracle $X$ halts on input $x$ and returns $y$, denoted $\ph_e^{X}(x) \downarrow = y$, by the formula $\ex {\sigma \subset X} \ph^{\sigma}_e(x) \downarrow = y$. Here $\ph^{\sigma}_e(x) \downarrow$ expresses the assertion that Turing machine $e$ halts on input $x$ in less than $\len(\sigma)$ steps, querying only the bits of $\sigma$. We can then define $e \in X'$ to mean $\ex {\sigma \subset X} \ph_e^\sigma(0) \downarrow$.

Now define $\ph_e^{X'}(x) \downarrow = y$ to mean $\ex {\sigma \subset X'} \ph_e^{\sigma}(x) \downarrow = y$. Using $\ax{B\Sigma_1}$, the assertion
$\sigma \subset X'$ is $\Delta^{0,X}_2$, as is the assertion $\ph_e^{X'}(x) = y$, assuming $\ph_e^{X'}$ is total.
We can then take $\ax{2\dash WWKL}$ to be the principle:
\begin{multline*}
 \fa {X,T} (\mbox{if $T$ is an infinite binary tree computable from $X'$ and} \\
   \lim_{n \to \infty} \frac{|\{\sigma \in T \st \len(\sigma) = n\}|}{2^n} > 0,
   \mbox{there is a path through $T$}).
\end{multline*}
We can express the second premise by saying that for some $\delta > 0$, for every $n$, there exists a finite set of $\sigma$'s of length $n$ in the tree making the sum greater than $\delta$. (More accurately, this expresses that the lim-inf of the expression in question is greater than $0$, which amounts to the same thing, as the expression is nonincreasing in $n$. But note that we do not assume that the limit exists.)

The first concern is also easily met by first considering $\Pi^{0,X}_2$ sets in a particularly nice form. Say that a $\Pi^{0,X}_2$ set $\bigcap_i \bigcup_j B_{i,j}$ is ``strict'' if the sets $\bigcup_j B_{i,j}$ are decreasing, in the sense that whenever $i' \geq i$, for every $j'$ there is a $j$ such that $B_{i',j'} \subseteq B_{i,j}$. Just as $\ax{B\Sigma_1}$ can be used to show that $\Sigma^0_1$ formulas are closed under bounded quantification, it can be used to show that $\Sigma^{0,X}_1$ sets are closed under finite intersections: in the identity
\[
  \bigcap_{i' \leq i} \bigcup_j B_{i',j} = \bigcup_{(j_0, \ldots, j_i)} (B_{0,j_0} \cap \cdots \cap B_{i,j_i}), 
\]
$\ax{B\Sigma_1}$ proves the left-to-right inclusion. Thus $\na{RCA_0}$ proves that every $\Pi^{0,X}_2$ set $\bigcap_i G_i$ is extensionally equivalent to a strict $\Pi^{0,X}_2$ set $\bigcap_i (\bigcap_{i' \leq i} G_{i'})$, and we can interpret references to the measure of $\bigcap_i G_i$ in terms of the measure of its strict equivalent. If $A = \bigcap_i \bigcup_j B_{i,j}$ is a strict $\Pi^{0,X}_2$ subset of $2^\omega$, we can express $\mu(A) \geq r$ as $\fa {i, \varepsilon > 0} \mu(\bigcup_j B_{i,j}) > r - \varepsilon$. In particular, the assertion that $A$ has positive measure is equivalent to $\ex {\delta > 0} \fa i (\mu(\bigcup_j B_{i,j}) > \delta)$. This gives $\ax{2\dash POS}$:
\[
 \fa {X, A \in \Pi^{0,X}_2} (\mu (A) > 0 \limplies \ex Y (Y \in A)).
\]

Our first goal is to show that, over $\na{RCA_0}$, the principles $\ax{2\dash WWKL}$ and $\ax{2\dash POS}$ are equivalent. This requires checking that $\na{RCA_0}$ can prove some fundamental facts about algorithmic randomness.

In the absence of $\ax{B\Sigma_2}$, reasoning about computability relative to the Turing jump $X'$ of a set $X$ is delicate. For example, $\ax{B\Sigma_2}$ is needed to show that if $f(n)$ is computable relative to $X'$ then so are the course-of-values function $g(n) = (f(0),\ldots,f(n-1))$ and the function $h(n) = \max_{i < n} f(i)$. Fortunately, in the presence of $\ax{2\dash WWKL}$, we have $\ax{B\Sigma_2}$ as well.

\begin{proposition}
\label{rand:prop:zero}
 Over $\na{RCA_0}$, $\ax{2\dash WWKL}$ implies $\ax{B\Sigma_2}$.
\end{proposition}

\begin{proof}
  Arguing in $\na{RCA_0 + \ax{2\dash WWKL}}$, suppose $\fa {x < a} \ex y \ph(x,y)$ where $\ph$ is a $\Pi^0_1$ formula, possibly with parameters other than the one shown. Via pairing, we can assume there is only one set parameter, $X$. Pick $k$ such that $2^k \geq a$ and let $\sigma_0, \ldots, \sigma_{a-1}$ be the first $a$ binary sequences of length $k$. The idea is to build a tree computable from $X'$ that includes all children of each $\sigma_x$ until they reach a length greater than some $y$ satisfying $\ph(x,y)$. More precisely, let $T$ be the tree defined by putting a sequence $\tau$ in $T$ if and only if $\tau$ is an initial segment of $\sigma_x$, for some $x < a$, or $\tau$ properly extends $\sigma_x$ and $\fa{y < \len(\tau)} \lnot \ph(x,y)$. The hypothesis implies there is no path through the tree, so, by $\ax{2\dash WWKL}$, there is some level $b$ such that the density of nodes of length $b$ in $T$ is less than $2^{-k}$. For any $x < a$, this implies that there is a $y < b$ satisfying $\ph(x,y)$; otherwise, all extensions of $\sigma_x$ of length $b$ would be in $T$, and these have density $2^{-k}$.
\end{proof}

Joseph Miller has pointed out to us that this proof also establishes that the principle $\ax{POS}$ of \cite{kjos:hanssen:et:al:unp}, mentioned in the introduction, also implies $\ax{B\Sigma_2}$.

Our next task is to prove that, over $\na{RCA_0}$, $\ax{2\dash WWKL}$ and $\ax{2\dash POS}$ are equivalent. Essentially, this involves showing that conventional computability-theoretic constructions can be carried out in $\na{RCA_0}$, and that claims regarding measure can be verified in the restricted axiomatic setting.

\begin{proposition}
\label{rand:prop:one}
 $\na{RCA_0}$ proves the following. Let $T$ be a tree computable from $X'$ satisfying the hypothesis of $\ax{2\dash WWKL}$. Then the set of paths through $T$ is a $\Pi^{0,X}_2$ set with positive measure.
\end{proposition}

\begin{proof}
 Argue in $\na{RCA_0}$. Let $T$ be as above. By the Shoenfield limit lemma \cite[Theorem 2.6.1]{downey:hirschfeldt:10} there is a $0,1$-valued function $f(\sigma,m)$ computable from $X$ such that for every $\sigma$, $\lim_m f(\sigma,m)$ exists, and $\sigma \in T$ if and only if $\lim_m f(\sigma,m) = 1$. (The proof of the limit lemma can be carried out using $\ax{B\Sigma_1}$.)   

 Define a sequence of trees $(T_m)$ where $\sigma \in T_m$ if $f(\tau, m)=1$ for all $\tau \subseteq \sigma$. By $\ax{B\Sigma_1}$, we have $\sigma \in T$ if and only if for every $k$ there is an $m > k$ such that $\sigma \in T_m$. For each $n$, let $G_n = \bigcup \{ [\sigma] \st \len(\sigma) = n \land \ex {k > n} \sigma \in T_k \}$. Clearly the sequence $(G_n)$ is decreasing. By $\ax{B\Sigma_1}$, each $G_n$ contains $\bigcup \{ [\sigma] \st  \len(\sigma) = n \land \sigma \in T \}$, so $\mu(G_n) \geq \delta$ for each $n$. 

 Hence $\bigcap_n G_n$ is a strict $\Pi^{0,X}_2$ set with positive measure, and it suffices to show that for any $Y$, $Y \in \bigcap_n G_n$ if and only if $Y$ is a path through $T$. Suppose $Y$ is a path through $T$. Then for every $n$, $Y \upharpoonright n$ is in $T$, and hence in $Y \in G_n$. Conversely, suppose $Y$ is not a path through $T$. Then for some $n$, $Y \upharpoonright n$ is not in $T$. Hence, for some $m \geq n$, $Y \upharpoonright n$ is not in $T_k$ for any $k \geq m$. Then $Y \upharpoonright m$ is not in $T_k$ for any $k \geq m$, and so $Y$ is not in $G_m$.
\end{proof}

The next proposition is an effective version of inner regularity for $G_\delta$ sets. The statement refers to the measure of a $\Pi^{0,X'}_1$ set, but we can make sense of this by combining conventions we have already discussed. Specifically, a $\Sigma^{0,X'}_1$ set $A$ is given by a sequence $(B_i)_{i \in \omega}$ of basic open sets computable in $X'$, given, say, by a function with index $e$. Then an element $Y$ of $2^\omega$ is in $A$ if and only if for some $i$ and $\sigma$, $\ph_e^{X'}(i) = \sigma \subset Y$. Hence the expression $Y \in A$ is given by a $\Sigma_2$ formula in $Y$, $X$, and $e$, and we can interpret statements involving the measure of $A$ as before. As before, a $\Pi^{0,X'}_1$ set is just the complement of a $\Sigma^{0,X'}_1$ set.

\begin{proposition}
 \label{rand:prop:two}
 $\na{RCA_0 + \ax{B\Sigma_2}}$ proves the following. Suppose $A$ is a $\Pi^{0,X}_2$ set such that $\mu(A) \geq r$, and $\delta > 0$. Then there is a $\Pi^{0,X'}_1$ set $C \subseteq A$ such that $\mu(C) \geq r - \delta$. 
\end{proposition}

\begin{proof}
 Let $A = \bigcap_i \bigcup_j B_{i,j}$ be a strict $\Pi^{0,X}_2$ set, where each $B_{i,j}$ is a basic open set. Notice that for every $i$, there is a $K_i$ such that for every $J > K_i$, $\mu(\bigcup_{j \leq J} B_{i,j} \setminus \bigcup_{j \leq K_i} B_{i,j})< \delta / 2^{i+2}$; otherwise, $\Sigma_1$ induction implies that for every $n$ there is a $J$ such that $\mu(\bigcup_{j \leq J} B_{i,j}) > n \delta / 2^{i+2}$, which contradicts the fact that the measure is bounded by $1$. Let $f(i)$ be a function, computable from $X'$, which returns such an $i$.

 Since every finite union of basic open sets is clopen, $C = \bigcap_i \bigcup_{j \leq f(i)} B_{i,j}$ is a closed set, and can be expressed explicitly as a $\Pi^{0,X'}_1$ set $\bigcap_i C_i$. (In more detail, write each $\bigcup_{j \leq f(i)} B_{i,j}$ as an intersection $\bigcap_{j \leq g(i)} \overline D_{i,j}$, where $D_{i,j}$ is a basic open set, and $g(i)$ is computable from $X'$. One needs $\ax{B\Sigma_2}$ to verify that $g(i)$ has the expected properties. We can then write $C = \bigcap_i \bigcap_{j \leq g(i)} \overline D_{i,j}$.) To show $\mu(C) \geq r - \delta$, we need to show that for every $i'$, $\mu(\bigcap_{i \leq i'} C_i) \geq r - \delta$. By $\ax{B\Sigma_1}$, it suffices to show that for every $i'$, $\mu(\bigcap_{i \leq i'} \bigcup_{j \leq f(i)} B_{i,j}) > r - \delta$, because every intersection $\bigcap_{i \leq i'} C_i$ is a superset of a larger intersection of the form $\bigcap_{i \leq i'} \bigcup_{j \leq f(i)} B_{i,j}$.

 Fix $i'$. The hypothesis that $\mu(A) \geq r$ implies that for some $J_1$ large enough, $\mu(\bigcup_{j \leq J_1} B_{i',j}) > r - \delta / 2$. The strictness of $A$ implies that $\bigcup_{j \leq J_1} B_{i',j}$ is included in $\bigcup_j B_{i,j}$ for each $i \leq i'$, and $\ax{B\Sigma_1}$ then implies there is a $J$ such that it is included in $\bigcup_{j \leq J} B_{i,j}$ for each $i \leq i'$. Hence $\mu(\bigcap_{i \leq i'} \bigcup_{j \leq J} B_{i,j}) > r - \delta / 2$.

 But now we are reduced to manipulations with finite unions and intersections. We have
\[
\bigcap_{i \leq i'} \bigcup_{j \leq J} B_{i,j} \setminus \bigcap_{i \leq i'} \bigcup_{j \leq f(i)} B_{i,j} \subseteq \bigcup_{i \leq i'} \left(\bigcup_{j \leq J} B_{i,j} \setminus \bigcup_{j \leq f(i)} B_{i,j}\right)
\]
 and the measure of this last set is less than $\sum_{i \leq i'} \delta / 2^{i+2} < \delta / 2$. Hence we have $\mu(\bigcap_{i \leq i'} \bigcup_{j \leq f(i)} B_{i,j}) > r - \delta$, as required.
\end{proof}

\begin{proposition}
\label{two:wwkl:pos:equiv}
 Over $\na{RCA_0}$, $\ax{2\dash WWKL}$ and $\ax{2\dash POS}$ are equivalent.
\end{proposition}

\begin{proof}
 First, assume $\ax{2\dash POS}$, and let $T$ be any tree computable from $X'$ that satisfies the hypothesis of $\ax{2\dash WWKL}$. By Proposition~\ref{rand:prop:one} the set of paths through $T$ is a $\Pi^{0,X}_2$ set with positive measure, and so, by $\ax{2\dash POS}$, has an element. 

  In the other direction, assume $\ax{2\dash WWKL}$, and let $A = \bigcup_i G_i$ be any strict $\Pi^{0,X}_2$ set such that $\mu(A) \geq \delta > 0$. Using Propositions~\ref{rand:prop:zero} and \ref{rand:prop:two}, there is a $\Pi^{0,X'}_1$ set $C \subseteq A$ such that $\mu(C) \geq \delta / 2$. By the usual reduction of closed sets to trees \cite{yu:simpson:90} we get a tree $T$ satisfying the hypotheses of $\ax{2\dash WWKL}$, such that any path through $T$ is an element of $C$ and hence $A$.
\end{proof}

We now turn to formalized notions of 2-randomness. Within the language of second-order arithmetic, define a \emph{$\Sigma^{0,X}_2$-test} to be a uniformly computable sequence $(A_i)_{i \in \omega}$ of $\Sigma^{0,X}_2$ sets such that for each $i$, $\mu(A_i) \leq 2^{-i}$. A set $Y$ is \emph{2-random relative to $X$} if for every $\Sigma^{0,X}_2$ test $(A_i)$ relative to $X$, $Y \not\in \bigcap_i A_i$, that is, $\ex i (Y \not\in A_i)$. Let $\ax{2\dash RAN}$ be the following principle:
\[
 \fa X \ex Y (\mbox{$Y$ is 2-random relative to $X$}).
\]
Once again, the fact that for every $X$ there is a universal $\Sigma^{0,X}_2$ test $(U_i)$ can be proved straightforwardly in $\na{RCA_0}$.

\begin{proposition}
\label{wwkl:ran:prop}
Over $\na{RCA_0}$, $\ax{2\dash WWKL}$ is equivalent to $\ax{2\dash RAN} +\ax{B\Sigma_2}$.
\end{proposition}

\begin{proof}
  By Proposition~\ref{two:wwkl:pos:equiv} we can use $\ax{2\dash POS}$ in place of $\ax{2\dash WWKL}$. In the forward direction, we already know that $\ax{2\dash WWKL}$ implies $\ax{B\Sigma_2}$, by Proposition~\ref{rand:prop:zero}. To obtain $\ax{2\dash RAN}$, fix $X$, and let $(U_i)_{i \in \omega}$ be a universal $\Sigma^{0,X}_2$ test. Then the complement of $U_1$ is a $\Pi^{0,X}_2$ set with positive measure, so by $\ax{2\dash POS}$, there is an element $Y$ in the complement of $U_1$, which is hence 2-random relative to $X$.

  For the other direction, let us show that $\na{RCA_0 + \ax{B\Sigma_2} + \ax{2\dash RAN}}$ proves $\ax{2\dash POS}$. Fix a $\Pi^{0,X}_2$ set $A$ with positive measure. By Proposition~\ref{rand:prop:two}, there is a subset $B$ of $A$ which is a $\Pi^{0,X'}_1$ set of positive measure. Let $Y$ be 2-random relative to $X$. As in the proof of Theorem~\ref{eq:thm:one}, $\ax{B\Sigma_2}$ implies that for some $i$, $\pi_i Y$ is in $B$ and hence $A$.
\end{proof}

The main conclusions of this section are summarized as follows:

\begin{theorem} 
\label{main:rand:theorem} 
Over $\na{RCA_0}$, the following are equivalent:
\begin{enumerate}
 \item $\ax{2\dash WWKL}$ 
 \item $\ax{2\dash POS}$
 \item $\ax{B\Sigma_2} + \ax{2\dash RAN}$
\end{enumerate}
\end{theorem}
The results of Slaman and Conidis mentioned at the end of the introduction then imply that all these principles are conservative over $\na{RCA_0 + \ax{B\Sigma_2}}$ for $\Pi^1_1$ sentences. Slaman and Conidis have posed the question as to whether $\ax{2\dash RAN}$ implies $\ax{B\Sigma_2}$ over $\na{RCA_0}$, or, equivalently, whether $\ax{2\dash RAN}$ implies $\ax{2\dash WWKL}$.

Recall that $\ax{WKL}$ is the axiom that asserts that every infinite tree on $\{0,1\}$ has a path. 

\begin{theorem}
\label{hier:thm}
Over $\na{RCA_0}$, each of $\ax{2\dash WWKL}$ and $\ax{WKL}$ implies $\ax{WWKL}$, but both implications are strict. Moreover, $\ax{2\dash WWKL}$ doesn't imply $\ax{WKL}$, and $\ax{WKL}$ doesn't imply $\ax{2\dash WWKL}$.
\end{theorem}

\begin{proof}
To see that $\ax{WKL}$ doesn't imply $\ax{2\dash WWKL}$, notice that the low basis theorem implies that there is an $\omega$-model of $\ax{WKL}$ in which every set is low, and so, in particular, $\Delta^0_2$ \cite{hajek:pudlak:93,simpson:99}. On the other hand, no $\Delta^0_2$ set is 2-random (see \cite[Section 6.8]{downey:hirschfeldt:10}). Finally, the model $\mdl M$ constructed by Yu and Simpson \cite[Section 2]{yu:simpson:90} satisfies $\na{RCA_0 + \ax{2\dash WWKL}}$, but not $\ax{WKL}$. (In fact, that model satisfies more; see Section~\ref{n:randomness:section}.)
\end{proof}

\section{Convergence theorems and reverse mathematics}
\label{reverse:section}

We now turn to the formalization of convergence theorems in reverse mathematics. We need to consider measures on an arbitrary compact separable metric space, along the lines of \cite{simpson:99,yu:93,yu:94,yu:simpson:90}. (The development there is closely related to the treatment of these notions in computable analysis, along the lines of \cite{hoyrup:rojas:09b,weihrauch:99}.) In particular, a compact separable metric space $(X, d)$ is assumed to be represented by a countable dense set of ideal points, and elements of the space are named by Cauchy sequences with an explicit rate of convergence. Open and closed sets, and more generally $\Sigma^{0,X}_n$ and $\Pi^{0,X}_n$ sets for any $X$ and $n$, are defined as in Section~\ref{randomness:section}, where now the basic open sets $B_i$ are balls $B(a,\delta)$, where $a$ is an ideal point and $\delta$ is rational. The space $C(X,\RR)$ of continuous functions on $X$ can be represented as the closure of a set of particularly simple ``test functions'' (called ``polynomials'' in \cite{simpson:99,yu:93,yu:94,yu:simpson:90}) under the uniform norm. 

A finite measure $\mu$ on $(X, d)$ is given in terms of the values of the integral $\int f \; d\mu$ on test functions; the space $L^1(\mu)$ is then defined as in Section~\ref{randomness:section}, replacing the simple functions there with test functions. Without loss of generality, we will assume $\mu(X) = 1$. Operations like pointwise addition and integration are defined straightforwardly and their basic properties can be established in $\na{RCA_0}$. Also, if $(f_n)$ is the name of an element $f$ of $L^1(\mu)$ and $k$ is any rational number, $(\min(f_n,k))$ and $(\max(f_n,k))$ are names for $\max(f,k)$ and $\min(f,k)$. Below we will rely on the fact that in $\na{RCA_0}$ these operations have the expected properties, such as $\min(f,k) + \max(f,k) = f + k$.

The measure of an open set $A$ is now defined as the supremum of the measure of test functions that are bounded by 1 and vanish outside of $A$. In contrast to the case with $2^\omega$, the measure of a basic open set need not be computable from $\mu$. But, as in Section~\ref{randomness:section}, if $A$ is $\Sigma^{0,X}_1$, the predicate $\mu(A) > \delta$ is $\Sigma^{0,X}_1$ definable. Randomness notions from Section~\ref{randomness:section} are easily adapted to this more general setting; see \cite{hoyrup:rojas:09b,simpson:99,yu:93,yu:94,yu:simpson:90} for details.

Let $\ax{G_\delta\dash POS}$ be the axiom that says that for every compact separable metric space $(X, d)$ and measure $\mu$, every $\Pi^{0,X}_2$ set with positive measure has an element. To link this up with axioms discussed in Section~\ref{formalization:section}, we only need the following two propositions. The first generalizes Proposition~\ref{rand:prop:two}.

\begin{proposition}
 \label{general:pos:prop}
 $\na{RCA_0 + \ax{B\Sigma_2}}$ proves the following. Fix a measure $\mu$ on a compact separable metric space $(X, d)$. Suppose $A$ is a $\Pi^{0,X}_2$ set such that $\mu(A) \geq r$, and $\delta > 0$. Then there is a $\Pi^{0,X'}_1$ set $C \subseteq A$ such that $\mu(C) \geq r - \delta$. 
\end{proposition}

\begin{proof}
The proof is similar to that of Proposition~\ref{rand:prop:two}, but slightly complicated by the fact that now basic open sets are no longer clopen. The construction we describe below is essentially that used by Yu~\cite[Lemma 4.1]{yu:93} to show that $\na{ACA_0}$ proves the regularity of measures for $G_\delta$ sets; we only need to confirm that the construction is computable in the Turing jump of the original set, and that the correctness can be verified in $\na{RCA_0 + \ax{B\Sigma_2}}$.

Suppose $A = \bigcap_i G_i$ where $G_i$ is a decreasing sequence of open sets and for each $i$, $\mu(G_i) \geq r$. Yu \cite[Lemma 2.3]{yu:93} notes that to each $G_i$ we can associate an increasing sequence of test functions $(g_{i,k})_{k \in \omega}$ that all vanish outside of $G_i$, with the property that whenever $\mu(G_i) > s$ then $\int g_{i,k} > s$ for sufficiently large $k$. 

As in the proof of Proposition~\ref{rand:prop:two}, we can define a function $f(i)$ computable in $X'$ such that for every $j \geq f(i)$, $\int g_{i,j} - \int g_{i,f(i)} < \delta / 2^{i+2}$. For each $i'$, let 
\[
 C_{i'} = \{ x \st \min_{i \leq i'} g_{i,f(i)}(x) \geq \delta / 4 \}.
\]
Using $\ax{B\Sigma_2}$, we can express this as a closed set computable in $X'$. (The collection axiom is needed to transport the bounded quantifier corresponding to the bounded minimization.) Let $C = \bigcap_{i'} C_{i'}$. If $x \in C_{i'}$ then $x \in \bigcap_{i \leq i'} G_i$, so $C \subseteq A$. Hence it suffices to show that for every $i'$, $\mu(C_{i'}) \geq r - \delta$.

For each $i'$, we have $\mu(C_{i'}) \geq \int \min_{i \leq i'} g_{i,f(i)} - \delta / 4$. As in the proof of Proposition~\ref{rand:prop:two} we can find a $J$ large enough so that $\int \min_{i \leq i'} g_{i,J} > r - \delta / 4$. But we also have
\[
\int \min_{i \leq i'} g_{i,J} -\int \min_{i \leq i'} g_{i,f(i)} \leq \sum_{i \leq i'} \left(\int g_{i,J} -\int g_{i,f(i)}\right) \leq \sum_{i \leq i'} \delta / 2^{i+1} < \delta / 2,
\]
so $\mu(C_{i'}) \geq r - \delta / 4 - \delta / 2 - \delta / 4 = r - \delta$, as required.
\end{proof}

\begin{proposition}
 \label{general:eq}
 Over $\ax{RCA_0}$, $\ax{G_\delta\dash POS}$ is equivalent to $\ax{2\dash WWKL}$.
\end{proposition}

\begin{proof}
 Given the usual representation of Cantor space as a metric space \cite{simpson:99}, $\ax{2\dash POS}$ is an instance of $\ax{G_\delta\dash POS}$. Thus by Theorem~\ref{main:rand:theorem} $\ax{G_\delta\dash POS}$ implies $\ax{2\dash WWKL}$, and we only have to prove the converse. Fortunately, Yu and Simpson \cite{yu:simpson:90} have done most of the hard work: Theorem 2 of \cite{yu:simpson:90} relativizes to show $\ax{2\dash WWKL}$ proves that every $\Pi^{0,X'}_1$ set with positive measure contains an element. By the previous proposition, every $\Pi^{0,X}_2$ set with positive measure contains a $\Pi^{0,X'}_1$ set with positive measure, so we are done.
\end{proof}

When $(f_n)$ is a sequence of functions, we can express the fact that $(f_n)$ converges to $f$ pointwise almost everywhere by saying that there is a null $G_\delta$ set $C$ such that for each point $x \not\in C$, each $f_n(x)$ is defined (see Section~\ref{formalization:section}) and $(f_n(x))$ converges to $f(x)$. Recall that in ordinary mathematics, a sequence $(f_n)$ converges to $f$ \emph{almost uniformly} if for every $\lambda > 0$ and $\varepsilon > 0$ there is an $n$ such that $\mu(\{ x \st \ex {m \geq n} |f_m(x) - f(x)| > \varepsilon\}) < \lambda$. In the case where each $f_n$ is a test function and $f(x) = 0$, the set in question is an open set, so when $(f_n)$ is a sequence of test functions we can straightforwardly express the fact that $(f_n)$ approaches $0$ almost uniformly in the language of second-order arithmetic. 

\begin{theorem}
 \label{main:thm}
Over $\na{RCA_0}$, the following are equivalent:
\begin{enumerate}
 \item $\ax{2\dash WWKL}$
 \item ``If $(f_n)$ is a sequence of test functions that converges to $0$ pointwise~a.e., then $(f_n)$ converges to $0$ almost uniformly.''
 \item $\ax{DCT'}$, that is, ``If $f$ and $g$ are elements of $L^1(\mdl X)$, $(f_n)$ is a sequence of elements of $L^1(\mdl X)$ dominated by $g$, and $(f_n)$ converges pointwise a.e.~to $f$, then $(\int f_n)$ converges to $\int f$.''
 \item ``If $(f_n)$ is a sequence of nonnegative test functions that is dominated by $1$ and converges to $0$ everywhere, then $(\int f_n)$ converges to $0$.''
\end{enumerate}
\end{theorem}

Notice that statement (2) is Egorov's theorem restricted to test functions. Statement (4) is the dominated convergence theorem with additional restrictions: the functions in the sequence are test functions, they are nonnegative and uniformly dominated by 1, and they converge everywhere (rather than just a.e.).

To prove that (1) implies (2), suppose $(f_n)$ is a sequence of test functions that converges to $0$ pointwise~a.e., but does not converge to $0$ almost uniformly. Then for some $\varepsilon > 0$ and $\lambda > 0$ and any $n$, if we set $A_n = \{ x \st \ex {m \geq n} |f_m(x)| > \varepsilon\}$, we have $\mu(A_n) > \lambda$. Since we are assuming $f_n$ converges a.e., there is an open set $B$ such that $\mu(B) < \lambda / 2$ and $(f_n)$ converges to $0$ off of $B$. Then $\bigcap_n A_n \setminus B$ is a $\Pi^{0,(f_n)}_2$ set with measure greater than $\lambda$, and so, by $\ax{2\dash POS}$, has an element $x$. But then $x \not\in B$ implies that $(f_n(x))$ converges to $0$ and $x \in \bigcap_n A_n$ implies that for every $n$ and there in $m \geq n$ such that $|f_m(x)| > \varepsilon$, a contradiction.

To prove that (2) implies (3), without loss of generality we will assume that $f = 0$ in the statement of $\ax{DCT'}$ and each $f_n$ and $g$ are nonnegative. First let us prove that $\int f_n$ converges to $0$ in the special case where $g$ is the constant function $1$. As in the proof of Theorem~\ref{effective:convergence:thm}, we can assume that each $f_n$ is a test function. By (2), we have that $(f_n)$ converges to $0$ almost uniformly, so, in particular, for any $\varepsilon > 0$ there is an $n$ such that for every $m \geq n$, $\mu(\{ x \st |f_m(x)| > \varepsilon / 2 \} ) < \varepsilon / 2$. Write
\[
\int f_m = \int \min(f_m,\varepsilon/2) + \int (\max(f_m,\varepsilon / 2) - \varepsilon / 2).
\] 
The first term is less than or equal to $\varepsilon / 2$. The function $\max(f_m,\varepsilon / 2) - \varepsilon / 2 = \max(f_m - \varepsilon / 2, 0)$ is bounded by 1 and vanishes outside $\{ x \st |f_m(x)| > \varepsilon / 2 \}$, so its integral is less than $\varepsilon / 2$. Hence $\int f_m < \varepsilon$, as required.

Scaling, we have that (2) implies (3) in the special case where $g$ is any constant function. To handle the more general case, we need a lemma. For any constant $K$ and nonnegative $f \in L^1(2^\omega)$, write $f^K$ for $\min(f,K)$. The next lemma shows that as $K$ approaches infinity, $f^K$ approaches $f$ in the $L_1$ norm, provably in $\na{RCA_0}$.

\begin{lemma}
\label{dct:lemma}
$\na{RCA_0}$ proves that if $g$ is any nonnegative element of $L^1(2^\omega)$ and $\varepsilon > 0$, there is an integer $K$ such that $\int (g - g^K) < \varepsilon$.
\end{lemma}

\begin{proof}
 Given $g$, choose a test function $g_n$ such that $\| g - g_n \| < \varepsilon$. Choose $K$ big enough so that $g_n$ is bounded by $K$. Then $\|g - g^K\| \leq \| g - g_n \| < \varepsilon$.

(Intuitively,  this last formula holds because $g(x) - g^K(x) = 0$ when $g(x) \leq K$, and $g(x) - g^K(x) < g(x) - g_n(x)$ when $g(x) > K$. Formally, one can show $\| g - g^K \| = \| \max(g,K) - K \| = \| \max(g - K, 0) \| \leq \| \max(g - g_n,0) \| \leq \| g - g_n \|$.)
\end{proof}

To complete the proof that (2) implies (3), now suppose $(f_n)$ approaches 0 pointwise~a.e.\ and is dominated by $g$. We need to show that for every $\varepsilon > 0$, there is an $m$ such that $\int f_n < \varepsilon$ for all $n \geq m$. Choose $K$ as in Lemma~\ref{dct:lemma} with $\varepsilon / 2$ in place of $\varepsilon$. Then $(f_n^K)$ still converges to 0 pointwise~a.e., so by the version of the dominated convergence theorem we have already proved, there is an $m$ such that for every $n \geq m$, $\int f_n^K < \varepsilon / 2$. But for every $n$, $\int (f_n - f_n^K) < \int (g - g_n^K) < \varepsilon / 2$, so for every $n \geq m$ we have $\int f_n < \varepsilon$, as required.

Clearly (3) implies (4). To show that (4) implies (1), consider any $\Pi^0_2$ set $A = \bigcap_n G_n$ with measure greater than or equal to $\delta > 0$. As in the proof of Theorem~\ref{two:pos:thm}, for each $n$ we can find a test function $f_n$ with the property $\int f_n > \delta$ but $f_n$ vanishes outside of $G_n$. In particular, $(f_n)$ converges to $0$ outside of $A$. By (4), there is an $x$ such that $(f_n(x))$ does not converge to $0$. This element, $x$, must be in $A$.

This completes the proof of Theorem~\ref{main:thm}. We also have an analogous version for the principle $\ax{DCT^*}$ described in the introduction. Say that a sequence $(f_n)$ is ``almost uniformly Cauchy'' if for every $\lambda > 0$ and $\varepsilon > 0$ there is an $n$ such that $\mu(\{x \st \ex {m, m' \geq n} |f_m(x) - f_{m'}(x)| > \varepsilon\}) < \lambda$. In ordinary mathematics, this is clearly equivalent to being almost uniformly convergent, but it has the advantage here that it does not require any mention of limits. 

\begin{theorem}
Over $\na{RCA_0}$, the following are equivalent:
\begin{enumerate}
\item $\ax{2\dash WWKL}$
\item ``If $(f_n)$ is a sequence of test functions such that $f_n(x)$ is Cauchy for almost every $x$, then $(f_n)$ is almost uniformly Cauchy.''
\item $\ax{DCT^*}$, this is, ``If $g$ is an element of $L^1(\mdl X)$, $(f_n)$ is a sequence of elements of $L^1(\mdl X)$ dominated by $g$, and the sequence $(f_n(x))$ is Cauchy for almost every $x$, then $(\int f_n)$ is Cauchy.''
\item ``If $(f_n)$ is a sequence of nonnegative test functions that is dominated by $1$ and the sequence $(f_n(x))$ is Cauchy for every $x$, then $(\int f_n)$ is Cauchy.''
\end{enumerate}
\end{theorem}

\begin{proof}
We need only slight modifications to the proof of Theorem~\ref{main:thm}. To show that (1) implies (2), replace the sets $A_n$ in the previous proof by $\{x \st \ex {m, m' \geq n} |f_m(x) - f_m'(x)| > \varepsilon\}$. To show that (2) implies (3) in the special case where $g = 1$, note that since $(f_n)$ is almost uniformly Cauchy, for any $\varepsilon > 0$ there is an $n$ such that for every $m, m'$ greater than or equal to $n$, $\mu(\{x \st |f_m(x) - f_{m'}(x)| > \varepsilon / 2\}) < \varepsilon / 2$. Then argue as before that for such $m$ and $m'$, 
\[
  \left|\int f_m - \int f_{m'} \right| \leq \int |f_m(x) - f_{m'}(x)| < \varepsilon.
\]
The generalization to arbitrary $g$ is as before. That (3) implies (4) is immediate. To show that (4) implies (1), let $(f_n)$ be as in the previous proof, and define the sequence $(f'_n)$ by $f'_{2n} = f_n$, $f'_{2n+1} = f_n / 2$. Then $(\int f'_n)$ is not Cauchy, since consecutive elements differ by at least $\lambda / 2$. By (4), there is a point $x$ such that $(f'_n(x))$ is not Cauchy; but since $(f_n)$ converges to $0$ outside of $A$, this point has to be in $A$.
\end{proof}

\begin{corollary}
 Over $\na{RCA_0}$, $\ax{DCT'}$ and $\ax{DCT^*}$ are equivalent to each other, and are strictly stronger than $\ax{WWKL}$, strictly weaker than $\ax{ACA}$, and not comparable with $\ax{WKL}$.
\end{corollary}

\section{Formalizing $n$-randomness}
\label{n:randomness:section}

In this section we observe that notions related to 2-randomness treated in Section~\ref{formalization:section} can be generalized to the corresponding notions for $n$-randomness, yielding a hierarchy of theories below $\ax{ACA_0}$. This involves adapting the proofs in Section~\ref{formalization:section} to formalize a number of basic properties of $n$-randomness (see \cite[Section 6.10]{downey:hirschfeldt:10}). As was the case for $n=2$ in Section~\ref{formalization:section}, $\ax{B\Sigma_n}$ is needed to show that $\Sigma^{0,X}_n$ sets are closed under bounded intersection, and also to develop a reasonable theory of computability relative to the $n$th Turing jump $X^{(n)}$ of a set $X$. Here we only sketch the details in the hopes that they will prove useful.

For each $n$, say that a \emph{strict} $\Pi^{0,X}_{n+1}$ (resp.~$\Sigma^{0,X}_{n+1})$ set is given by a decreasing (resp.~increasing) sequence $A_0 \supseteq A_1 \supseteq A_2 \supseteq \ldots$ where each $A_n$ in turn is a strict $\Sigma^{0,X}_n$ (resp.~$\Pi^{0,X}_n$) set. If is not hard to show that $\na{RCA_0 + \ax{B\Sigma_n}}$ proves that every $\Pi^{0,X}_{n+1}$ set is extensionally equal to a strict one. 

If $\bigcup_n A_n$ is a strict $\Sigma^{0,X}_n$ set, then $\mu(\bigcup_n A_n) > \delta$ is defined by the $\Sigma^{0,X}_n$ formula $\ex n (\mu(A_n) > \delta)$. Similarly, if $B$ is a strict $\Pi^{0,X}_n$ set, $\mu(A) \geq \delta$ is defined by a $\Pi^{0,X}_n$ formula. By the results of Yu \cite{yu:93}, over $\na{ACA_0}$ this agrees with the definition of the measure of a set in terms of the infimum of the measures of the open sets covering it. The principles $\ax{n\dash WWKL}$, $\ax{n\dash POS}$, and $\ax{n\dash RAN}$ are then defined as in Section~\ref{formalization:section}. The results of Section~\ref{randomness:section} carry over, yielding:

\begin{theorem}
Let $n \geq 1$. Over $\na{RCA_0}$, the following are equivalent:
\begin{enumerate}
 \item $\ax{n\dash WWKL}$
 \item $\ax{n \dash POS}$
 \item $\ax{B\Sigma_n} + \ax{n\dash RAN}$.
\end{enumerate}
 \end{theorem}

Moreover, the proof of Proposition~\ref{general:eq} generalizes to show that $\ax{n\dash POS}$ is equivalent to the corresponding principle for arbitrary measures on compact measure spaces.

This gives rise to the following picture, in which the only implications that hold are the ones indicated. Note that the model $\mdl M$ of Yu and Simpson \cite{yu:simpson:90}, mentioned in the proof of Proposition~\ref{hier:thm}, satisfies all the principles $\ax{n\dash WWKL}$ but not $\ax{WKL}$.

\begin{center}
\begin{tikzpicture}[->, node distance=1.0cm, semithick]

 \node (RCA)                            {$\na{RCA_0}$};
 \node (WWKL)   [above of=RCA]          {$\ax{WWKL}$};
 \node (blank1) [above of=WWKL]         {};
 \node (WKL)    [right of=blank1]       {$\ax{WKL}$};
 \node (2WWKL)  [left of=blank1]        {$\ax{2\dash WWKL}$};
 \node (3WWKL)  [above of=2WWKL]        {$\ax{3\dash WWKL}$};
 \node (dots)   [above of=3WWKL,
                 node distance=0.65cm]   {$\vdots$};
 \node (blank2) [right of=dots]         {};
 \node (ACA)    [above of=blank2]       {$\ax{ACA}$};

 \draw (ACA)    -> (dots);
 \draw (ACA)    -> (WKL);
 \draw (3WWKL)  -> (2WWKL);
 \draw (2WWKL)  -> (WWKL);
 \draw (WKL)    -> (WWKL); 
 \draw (WWKL)   -> (RCA);
  
\end{tikzpicture}
\end{center}


\end{document}